\documentclass[11pt,english]{smfart}

\usepackage[english]{babel}

\usepackage{amssymb,url,xspace,smfthm}
\usepackage{tikz}
\usetikzlibrary{arrows}
\usepackage{hyperref}

\makeatletter
    \def\ps@copyright{\ps@empty
    \def\@oddfoot{\hfil\small\copyright 2012, \SMF}}
\makeatother

\newcommand{\SMF}{Soci\'et\'e Ma\-th\'e\-ma\-ti\-que de France}
\newcommand{\BibTeX}{{\scshape Bib}\kern-.08em\TeX}
\newcommand{\T}{\S\kern .15em\relax }
\newcommand{\AMS}{$\mathcal{A}$\kern-.1667em\lower.5ex\hbox
        {$\mathcal{M}$}\kern-.125em$\mathcal{S}$}

\usepackage[all]{xy}
\input yoga.sty

\title{Group schemes out of birational group laws, N\'eron models}
\date{\today}

\author[Edixhoven]{Bas Edixhoven}
\address[Edixhoven]{Mathematisch Instituut \\
Universiteit Leiden \\
Postbus 9512 \\
2300 RA Leiden \\
Nederland}
\email{edix@math.leidenuniv.nl}

\author[Romagny]{Matthieu Romagny}
\address[Romagny]{Institut de Recherche Math\'ematique de Rennes \\
Universit\'e Rennes 1 \\
Campus de Beaulieu \\
35042 Rennes Cedex \\
France}
\email{matthieu.romagny@univ-rennes1.fr}


\begin{document}

\begin{abstract}
In this note, we present the theorem of extension
    of birational group laws in both settings of classical varieties
    (Weil) and schemes (Artin). We improve slightly the original proof
    and result with a more direct construction of the group extension,
a discussion of its separation properties,
    and the systematic use of algebraic spaces. We also explain the
    important application to the construction of N\'eron models of
    abelian varieties. This note grew out of lectures given by Ariane
    M\'ezard and the second author at the Summer School "Sch\'emas en
    groupes" held in the CIRM (Luminy) from 29 August to 9 September,
    2011.

\end{abstract}

\begin{altabstract}
Dans cette note, nous pr\'esentons le th\'eor\`eme d'extension d'une
loi de groupe birationnelle en un groupe alg\'ebrique,
dans le cadre des vari\'et\'es alg\'ebriques
classiques (Weil) et des sch\'emas (Artin). Nous am\'eliorons
l\'eg\`erement le r\'esultat original et sa preuve en donnant une
construction plus directe du groupe,
en apportant des compl\'ements sur ses propri\'et\'es de s\'eparation,
et en utilisant syst\'ematiquement les espaces alg\'ebriques.
Nous expliquons aussi l'application importante \`a la construction
des mod\`eles de N\'eron des vari\'et\'es ab\'eliennes. Cette note
est issue des cours donn\'es par Ariane M\'ezard et le second auteur
\`a l'\'Ecole d'\'et\'e "Sch\'emas en groupes" qui s'est tenue au
CIRM (Luminy) du 29 ao\^ut au 9 septembre 2011.
\end{altabstract}

\maketitle


\vskip-2cm

\section{Introduction}
This paper is devoted to an exposition of the generalization to group
schemes of Weil's theorem in \cite{Wei2} on the construction of a group
from a birational group law, as can be found in Artin's Expos\'e XVIII
in SGA3~\cite{Art}.  In addition, we show how this theorem is used by
N\'eron in order to produce canonical smooth models (the famous {\em
  N\'eron models}) of abelian varieties.

The content of Weil's theorem is to extend a given ``birational group
law'' on a scheme~$X$ to an actual multiplication on a group
scheme~$G$ birational to~$X$. The original motivation of Weil was the
algebraic construction of the Jacobian varieties of
curves~\cite{Wei1}. This construction was extended by Rosenlicht to
generalized Jacobians~\cite{Ros}. Weil's ideas were later used by
Demazure in his thesis in order to show the existence of split
reductive groups over the ring $\mathbf{Z}$ of integers~\cite{Dem2}
and then by N\'eron in order to study minimal models of abelian
varieties~\cite{N}. To our knowledge, these are the three main
applications of the extension theorem.

The construction of split reductive groups by Demazure uses a version
of Weil's theorem written by Artin, valid for flat (maybe non-smooth)
finitely presented group schemes. There, the set-theoretic arguments
of Weil are replaced by sheaf-theoretic arguments.  The main point
then is to show that Weil's procedure gives a sheaf which is
representable; since this sheaf is defined as a quotient by an fppf
equivalence relation, the natural sense in which it is representable
is as an algebraic space (i.e. a quotient of an \'etale equivalence
relation of schemes, see~\ref{conventions_alg_spaces}). However, at the
time when Artin
figured out his adaptation of Weil's result, he had not yet discovered
algebraic spaces. Thus he had to resort at times to ad hoc statements;
for example, his main statement (Theorem~3.7 of \cite{Art}) is a bit
unsatisfying. Nowadays it is more natural to use the language of
algebraic spaces, and this is what we shall do. As an aside, it is
clear that one may as well start from a birational group law on an
{\em algebraic space}, but we do not develop this idea.

Another feature of Artin's proof is that he constructs~$G$ let us say
``in the void'', and that needs a lot of verifications that moreover
are not so structured.  We give a more structured proof of Theorem~3.7
of~\cite{Art}.  The idea is to construct the group space $G$ as a
subfunctor of the $S$-functor in groups~$\calR$ that sends~$T$ to the
group of $T$-birational maps from $X_T$ to itself, as in Section~5.2
of Bosch-L\"utkebohmert-Raynaud \cite{BLR}. We push the construction
of \cite{BLR} a bit further: we show that $\calR$ is a sheaf and we
define $G$ to be the subsheaf of groups generated by the image of $X$
under a morphism that sends $a$ in $X(T)$ to the rational
left-translation by $a$ on~$X_T$.

One technical detail is that whereas Artin requires $X$ to be of
finite presentation, we allow it to be only locally so (that is, maybe
not quasi-compact and quasi-separated). This turns out to need no
modification of our proofs, and may be interesting for instance for
the treatment of N\'eron models of semi-abelian varieties, since these
fail to be of finite type.

A significant difference between \cite{Art} and Section~5.2 of \cite{BLR} is
that \cite{BLR} treats descent only in Chapter~6, after the construction of
groups from birational ones. So, Chapter~5 of \cite{BLR} is more geometric
and less sheaf-theoretic than \cite{Art}. It is a good thing to compare
the two accounts. Here are some considerations.

\begin{enumerate}
\item In \cite{Art}, $S$ is arbitrary, and $X/S$ is faithfully flat and
  of finite presentation, with separated fibres without embedded
  components. The conclusion is that $G/S$ is an algebraic space.
\item In Theorem~5.1/5 of~\cite{BLR}, the scheme $S$ is the spectrum of a
  field or of a discrete valuation ring, and $X/S$ is separated,
  smooth and quasi-compact, and surjective. 
\item In Theorem~6.1/1 of~\cite{BLR}, $S$ is arbitrary, $X/S$ is smooth,
  separated, quasi-compact. The conclusion is that $G/S$ is a
  scheme. For the proof of this theorem, whose main ideas come
from Raynaud~\cite{Ra2}, Theorem~3.7 of \cite{Art} is
  admitted, although it is also said that if $S$ is normal, then it
  can be obtained as in Chapter~5 of \cite{BLR}. 
\item In \cite{Art} the birational group law is
  ``strict''. Proposition~5.2/2 of~\cite{BLR} and \cite{Wei2} reduce,
  under certain conditions, the case of a birational group law to a
  strict one.
\end{enumerate}

Let us now briefly  describe what we say on the application to N\'eron models.
While N\'eron's original paper was written in the old language of Weil's
Foundations and quite hard to read, the book~\cite{BLR} is a modern
treatment that provides all details and more on this topic. It is
however quite demanding for someone who wishes to have a quick overview
of the construction. In this text, we tried to show to the reader that
it is in fact quite simple to see not only the skeleton but also almost
all the flesh of the complete construction. Thus we bring out the
main ideas of N\'eron to produce a model of the abelian variety one
started with, endowed with a strict birational group law. Then
Weil's extension theorem finishes the job. The few things that we do not
prove are:
\begin{enumerate}
\item the decreasing of N\'eron's measure for the defect of smoothness
under blow-up of suitable singular strata (Lemma~\ref{delta_decreases}),
\item the theorem of Weil on the extension of morphisms from smooth
schemes to smooth separated group schemes (proof of
Proposition~\ref{prop_univ_property}).
\end{enumerate}
In both cases, using these results as black boxes does not interrupt
the main line of the proof, and moreover there was nothing we could add
to the proofs of these facts in~\cite{BLR}.

The exposition of Weil's theorem occupies sections 2 and 3 of the paper,
while the application to N\'eron models occupies sections 4 to 6.

\section{A case treated by Andr\'e Weil}

Let $k$ be an algebraically closed field. An {\em algebraic variety} over
$k$ will mean a $k$-scheme that is locally of finite type, separated,
and reduced. For such an $X$, we denote $X(k)$ by $X$ itself, that
is, we forget about the non-closed points.
A subvariety of $X$ is said to be {\em dense} if it is topologically
dense.

Let, in this paragraph, $G$ be an algebraic variety over $k$ with an
algebraic group structure. Then the graph of the multiplication map
from $G\times G$ to $G$ is a closed subvariety $\Gamma$ of $G\times
G\times G$; it is the set of $(a,b,c)$ in $G\times G\times G$ such
that $c=ab$. For every $i$ and $j$ in $\{1,2,3\}$ with $i<j$ the
projection $\pr_{i,j}\colon \Gamma\to G\times G$ is an isomorphism,
hence $\Gamma$ is the graph of a morphism $f_{i,j}:=\pr_k\circ
\pr_{i,j}^{-1}$, where $\{i,j,k\}=\{1,2,3\}$, from $G\times G$ to
$G$. We have $f_{1,2}(a,b)=ab$, $f_{1,3}(a,c)=a^{-1}c$ and
$f_{2,3}(b,c)=cb^{-1}$. For $X$ a dense open subvariety of $G$ and $W$
a dense open subvariety of $\Gamma$ contained in $X\times X\times X$,
the pair $(X,W)$ is a strict birational group law as in the following
definition. Theorem~\ref{thm_Weil} shows in fact that each strict
birational group law is in fact obtained in this way.

\begin{defi}\label{def_2.1}
Let $X$ be an algebraic variety over $k$, not empty. A \emph{strict
  birational group law on $X$} is a subvariety (locally closed, by
definition) $W$ of $X\times X\times X$, that satisfies the following
conditions.
\begin{enumerate}
\item For every $i$ and $j$ in $\{1,2,3\}$ with $i<j$ the projection
  $\pr_{i,j}\colon W\to X\times X$ is an open immersion whose image,
  denoted $U_{i,j}$, is dense in $X\times X$. For each such $(i,j)$,
  we let $f_{i,j}\colon U_{i,j}\to X$ be the morphism such that $W$ is
  its graph. For every such $(i,j)$ and for every $x=(x_1,x_2,x_3)$ in
  $X^3$ the condition $x\in W$ is equivalent to: $(x_i,x_j)\in
  U_{i,j}$ and $x_k=f_{i,j}(x_i,x_j)$, with $\{i,j,k\}=\{1,2,3\}$. We
  denote the morphism $f_{1,2}\colon U_{1,2}\to X$ by $(a,b)\mapsto
  ab$. Hence, for $(a,b,c)$ in $X^3$ we have $(a,b,c)\in W$ if and
  only if $(a,b)\in U_{1,2}$ and $c=ab$.
\item For every $a$ in $X$, and for every $i$ and $j$ in $\{1,2,3\}$
  with $i<j$ the inverse images of $U_{i,j}$ under the morphisms
  $(a,\id_X)$ and $(\id_X,a)\colon X\to X\times X$ are dense
  in~$X$ (in other words, $U_{i,j}\cap (\{a\}\times X)$ is dense in
  $\{a\}\times X$ and $U_{i,j}\cap (X\times \{a\})$ is dense in
  $X\times \{a\}$). 
\item For all $(a,b,c)\in X^3$ such that $(a,b)$, $(b,c)$, $(ab,c)$
  and $(a,bc)$ are in $U_{1,2}$, we have $a(bc)=(ab)c$.
\end{enumerate}
\end{defi}

From now on, $X$ is an algebraic variety over $k$ with a strict
rational group law~$W$. The idea in what follows is that we can let
$X$ act on itself by left and right translations, which are rational
maps. Left translations commute with right translations, and the group
we want to construct can be obtained as the group of birational maps
from $X$ to $X$ that is generated by the left translations, or,
equivalently, the group of birational maps from $X$ to $X$ that
commute with the right translations.

\begin{defi}
We let $\calR$ be the set of birational maps from $X$ to itself, that
is, the set of equivalence classes of $(U,f,V)$, where $U$ and $V$ are
open and dense in $X$ and $f\colon U\to V$ is an isomorphism, where
$(U,f,V)$ is equivalent to $(U',f',V')$ if and only if $f$ and $f'$
are equal on $U\cap U'$ (note that $X$ is separated, this is needed
for transitivity of the relation). For each element $g$ of $\calR$
there is a maximal dense open subset $\Dom(g)$ of $X$ on which it is a
morphism.
\end{defi}

\begin{rema}
The elements of $\calR$ can be composed, they have inverses, and so
$\calR$ is a group. For $(U,f,V)$ as above, let $g$ be $f^{-1}\colon V\to
U$, then $f$ and $g$ induce inverse morphisms between $f^{-1}\Dom(g)$
and $g^{-1}\Dom(f)$, and therefore $(f^{-1}\Dom(g),f,g^{-1}\Dom(f))$
is a maximal representative of the equivalence class of $(U,f,V)$ (see
the proof of Lemma~\ref{lem_3.6} for details).
\end{rema}
\begin{lemm}
For $a$ in $X$, let $U_a:=(a,\id_X)^{-1}U_{1,2}$ and
$V_a:=(a,\id_X)^{-1}U_{1,3}$. Then $U_a$ and $V_a$ are open and dense
in~$X$, and $f_{1,2}\circ (a,\id_X)\colon U_a\to X$, $x\mapsto ax$, and
$f_{1,3}\circ(a,\id_X)\colon V_a\to X$ induce inverse morphisms
between $U_a$ and $V_a$. 
\end{lemm}
\begin{proof}
Let $a\in X$. For $b$ and $c$ in $X$, the condition $(a,b,c)\in W$ is
equivalent to ($(a,b)\in U_{1,2}$ and $c=f_{1,2}(a,b)$), and to
($(a,c)\in U_{1,3}$ and $b=f_{1,3}(a,c)$). But $(a,b)\in U_{1,2}$
means, by definition, that $b\in U_a$. And $(a,c)\in U_{1,3}$ means
that $c\in V_a$.  
\end{proof}
\begin{defi}
For $a$ in $X$, we let $\phi(a)$ denote the element of $\calR$ given
by $(U_a,f_{1,2}\circ (a,\id_X),V_a)$. Hence: $\phi(a)\colon U_a\to
V_a$ is the isomorphism $x\mapsto ax$. We have $\phi\colon X\to
\calR$, a map of sets, from $X$ to the group~$\calR$. We let $G$ be
the subgroup of $\calR$ generated by~$\phi(X)$. For $a$ in $X$, let
$\psi(a)$ be the element of $\calR$ given by $(U_a',f_{1,2}\circ
(\id_X,a),V_a')$, where $U_a'=(\id_X,a)^{-1}U_{1,2}$ and
$V_a'=(a,\id_X)^{-1}U_{2,3}$. Then $\psi(a)$ is the isomorphism
$x\mapsto xa$ from $U_a'$ to~$V_a'$. 
\end{defi}

\begin{lemm}\label{lem_phi_homomorphism}
For all $(a,b)\in U_{1,2}$ we have $\phi(a)\circ \phi(b)=\phi(ab)$.
For all $a$ and $b$ in $X$, we have $\phi(a)\circ\psi(b) =
\psi(b)\circ\phi(a)$. Every $g$ in $G$ commutes with every $\psi(b)$
($b\in X$).
\end{lemm}
\begin{proof}
The first two statements follow from the associativity of the
birational group law. The last statement follows from the
definition of $G$: it is generated by $\{\phi(a): a\in X\}$.
\end{proof}

\begin{lemm}\label{lem_2.7}
Let $g$ be in $G$ and $x\in\Dom(g)$ such that $g(x)=x$. Then $g=\id_X$.
\end{lemm}
\begin{proof}
For every $b\in X$ such that $b\in U_x$ and $xb\in\Dom(g)$, we have
\[
g(xb) = (g\circ\psi(b))x = (\psi(b)\circ g)x = xb.
\]
Hence $g(y)=y$ for all $y$ in a dense open subset of~$X$.
\end{proof}

\begin{lemm}
The map $\phi\colon X\to G$ is injective.
\end{lemm}
\begin{proof}
Let $a$ and $b$ be in $X$, such that $\phi(a) = \phi(b)$. Then, for
all $x\in U_a\cap U_b$, we have $(a,x,ax)\in W$ and $(b,x,bx)\in
W$. But these two points of $W$ have the same image under $\pr_{2,3}$,
as $ax=bx$. As $\pr_{2,3}$ is injective, $a=b$.
\end{proof}
In order to make $G$ into a group variety, the idea is now simply to
use translates of $\phi\colon X\to G$ as charts. 
\begin{defi}
For $g$ in $G$, let $\phi_g\colon X\to G$ be given by $a\mapsto
g\circ\phi(a)$. Note that $\phi_g$ is $\phi\colon X\to G$ followed by
left-multiplication by $g$ on~$G$.
\end{defi}
The $\phi_g$ cover $G$ because $X$ is not empty. The next lemma shows
that these charts are compatible: for $g_1$ and $g_2$ in $G$,
$\phi_{g_2}^{-1}(\phi_{g_1}(X))$ is open in $X$, and the map
$\phi_{g_1}^{-1}\circ\phi_{g_2}\colon\phi_{g_2}^{-1}\phi_{g_1}(X)\to
X$ sending $x$ to $\phi_{g_1}^{-1}(\phi_{g_2}(x))$, is a morphism. 

\begin{lemm}\label{lem_charts_compatible}
For $g_1$ and $g_2$ in $G$, the set of $(a,b)$ in $X^2$ such that
$\phi_{g_1}(a)=\phi_{g_2}(b)$ is the transpose of the graph of
$g_1^{-1}g_2\colon\Dom(g_1^{-1}g_2)\to X$, $b\mapsto (g_1^{-1}g_2)(b)$,
that is, the set $\{(a,b)\in X^2:\text{$b\in\Dom(g_1^{-1}g_2)$ and
  $a=(g_1^{-1}g_2)b$} \}$.
\end{lemm}
\begin{proof}
Let $g:=g_1^{-1}g_2$. We want to know for which $(a,b)\in X^2$ the
condition $\phi(a)=g\circ\phi(b)$ holds. Let $(a,b)$ be in $X^2$. 

Assume that $\phi(a)=g\circ\phi(b)$. Then there is an $x$
in $X$ such that $(a,x)\in U_{1,2}$, $(b,x)\in U_{1,2}$, $(b,x)\in
f_{1,2}^{-1}\Dom(g)$, and $ax=g(bx)$. Let $x$ be such. 

Then $a\in\Dom(\psi(x))$ and $(\psi(x))(a)=ax$, $b\in\Dom(\psi(x))$
and $(\psi(x))(b)=bx$, and $(\psi(x))(b)=bx\in\Dom(g)$, 
and $(\psi(x))(a)=g((\psi(x))(b))$, and
$(\psi(x))(a)=ax\in\Dom(\psi(x)^{-1})$. 

Hence $b\in\Dom(\psi(x))$, $(\psi(x))(b)\in\Dom(g)$,
$g((\psi(x))(b))\in\Dom(\psi(x)^{-1})$, and 
$a=(\psi(x)^{-1}\circ g\circ\psi(x))(b)$. 
Now note that $\psi(x)^{-1}\circ g\circ\psi(x)=g$ in $\calR$, hence 
$b\in\Dom(g)$ and $a=g(b)$.

Now assume that $b\in\Dom(g)$ and $a=g(b)$. We must prove that
$\phi(a)=g\circ\phi(b)$ in~$G$. 

Let $x$ be in $X$ such that $(a,x)\in U_{1,2}$, $(b,x)\in U_{1,2}$,
and $bx\in\Dom(g)$.  Then $(\phi(a))(x)=ax$ because $(a,x)\in
U_{1,2}$. We have $ax=g(b)x$, because $a=g(b)$. We have
$(g(b)x)=\psi(x)(g(b))$ because $(g(b),x)\in U_{1,2}$. We have
$\psi(x)(g(b))=(\psi(x)\circ g)(b)$ because $b\in\Dom(g)$, $g(b)=a$
and $(a,x)\in U_{1,2}$. We have $\psi(x)\circ g=g\circ\psi(x)$ in
$\calR$, hence $(\psi(x)\circ g)(b)=(g\circ\psi(x))(b)$. We have
$(g\circ\psi(x))(b)=g((\psi(x))(b))$ because $b\in\Dom(\psi(x))$ and
$(\psi(x))(b)\in\Dom(g)$. We have
$g((\psi(x))(b))=g(bx)=g((\phi(b))(x))=(g\circ\phi(b))(x)$. We
conclude, using Lemma~\ref{lem_2.7}, that $\phi(a)=g\circ\phi(b)$
in~$G$. 
\end{proof}

\begin{theo}\label{thm_Weil}
Let $k$ be an algebraically closed field. Let $X$ be an algebraic
variety over $k$, that is, the variety of $k$-points of a $k$-scheme
that is locally of finite type, separated and reduced. Let~$W$ be a
strict birational group law on~$X$. Let $G$ be the group constructed
as above.
\begin{enumerate}
\item
The charts $\phi_g\colon X\to G$ are compatible, and each of
  them is an open immersion with dense image. They give $G$ the
  structure of a $k$-scheme that is reduced, and locally of finite
  type. 
\item
The group law on $G$ extends the birational group law on $X$
  that is given by~$W$. As a $k$-group scheme, locally of finite type,
  $G$ is separated.
\item
The map $X\times X\to G$, $(a,b)\mapsto \phi(a)\phi(b)^{-1}$ is
surjective, and the fibre over $g$ in $G$ is the set
$\{(a,b)\in X^2: \text{$b\in\Dom(g)$ and  $a=g(b)$}\}$.
\item Every $g$ in $G$ induces an isomorphism between $\Dom(g)$ and
  $\Dom(g^{-1})$. 
\item If $X$ is of finite type, then $G$ is of finite type.
\end{enumerate}
\end{theo}
\begin{proof}
For 1, apply Lemma~\ref{lem_charts_compatible}.

For 2, use Lemma~\ref{lem_phi_homomorphism}, and that the diagonal
in $G\times G$ is the fibre over the unit element of the morphism
$G\times G\to G$, $(x,y)\mapsto x^{-1}y$.

For 3, use Lemma~\ref{lem_charts_compatible}.

For 4, use 3.

For 5, note that $G$ is locally of finite type, and, as the image of
$X^2$, quasi-compact.
\end{proof}

\begin{rema} \label{rem_extensions}
Let $A,B$ be two commutative connected algebraic groups. Let
$\mathrm{Ext}(A,B)$ be the set of classes of classes of extensions of
$A$ by $B$.  A rational map $f\colon A\times A\longrightarrow B$
satisfying the identity
\[f(y,z)-f(x+y,z)+f(x,y+z)-f(x,y)=0\]
for all $x,y,z\in A$ is called a {\em rational factor system}.
It is called {\em
  symmetric} if $f(x,y)=f(y,x)$, $x,y\in A$.  Such a system will be
called {\em trivial} if there exists a rational map
$g\colon A\longrightarrow B$ such that $f=\delta g$, where
\[
\delta g(x,y)=g(x,y)=g(x+y)-g(x)-g(y).
\]
The classes of symmetric factor
systems modulo the trivial factor systems form a group denoted
$H^2_{\rm rat}(A,B)_s$. As an application of Weil's theorem, one can
show that $H^2_{\rm rat}(A,B)_s$ is isomorphic to the subgroup of
$\mathrm{Ext}(A,B)$ given by the extensions which admit a rational
section (see \cite{Ser},Chap.~VII, \S~1, no~4., Prop.~4).
\end{rema}

\section{The case treated by Michael Artin in SGA3}

In order to state and prove a relative version of the extension
theorem for birational group laws, a short reminder on $S$-rational
maps is useful. Definitions~\ref{defs_3.1} and
Proposition~\ref{prop_3.2} below can be found, some parts in a broader
generality, in \cite{Art}, Section~1, except for the statements where
other references are given.

\begin{defi}\label{defs_3.1}
Let $S$ be a scheme. Let $X$, $Y$ be $S$-schemes. 
\begin{enumerate}
\item An open subscheme $U\subset X$ is {\em $S$-dense} if $U\times_S
  S'$ is schematically dense in $X\times_S S'$ for all morphisms
  $S'\to S$.
\item An {\em $S$-rational map} $f\colon X\dashrightarrow Y$ is an
  equivalence class of morphisms $U\to Y$ with $U\subset X$ open and
  $S$-dense, where $U\to Y$ and $V\to Y$ are equivalent if they agree
  on an $S$-dense open subscheme $W\subset U\cap V$.
\item An {\em $S$-birational map} is an $S$-rational map that can be
  represented by a morphism $U\to Y$ inducing an isomorphism with an
  $S$-dense open subscheme of~$Y$.
\end{enumerate}
\end{defi}

\begin{prop}\label{prop_3.2}
Let $S$ be a scheme. Let $X$ and $Y$ be $S$-schemes that are flat and
locally of finite presentation. 
\begin{enumerate}
\item By \cite{EGA}~IV.11.10.10, an open subscheme $U$ of $X$ is
  $S$-dense if and only if for all $s\in S$, $U_s$ is schematically
  dense in $X_s$ (that is, $U_s$ contains the associated points of
  $X_s$).
\item Unions of non-empty families and finite intersections of $S$-dense opens are
  $S$-dense open.
\item If $U\subset X$ and $V\subset Y$ are $S$-dense open subschemes,
  then $U\times_S V$ is an $S$-dense open subscheme of~$X\times_S Y$.
\item Let and $f$ and $g$ be $S$-morphisms from $X$ to~$Y$. Assume
  that the fibers of $Y\to S$ are separated, and that $f$ and $g$ are
  equal on an $S$-dense open subscheme $U$ of~$X$. Then $f=g$.
\item Let $f\colon X\dashrightarrow Y$ be an $S$-rational map. Assume
  that the fibres of $Y\to S$ are separated. Then there is a maximal
  $S$-dense open subscheme $U\subset X$ with a morphism $U\to Y$
  representing~$f$, called the {\em domain of definition} of~$f$ and
  denoted $\Dom_S(f)$. Its reduced complement is called the {\em
    exceptional locus} of~$f$. For $S'\rightarrow S$ flat and locally
  of finite presentation, $\Dom_{S'}(f\times_S S')=\Dom_S (f)\times_S S'$.
\end{enumerate}
\end{prop}

From now on and until the rest of this section, we put ourselves in
the situation of Theorem~3.7 of~\cite{Art}, namely we make the
following:

\begin{enonce}{Assumptions}
Let $S$ be a scheme, and $f\colon X\to S$ be an $S$-scheme that is
faithfully flat and locally of finite presentation, whose fibres are
separated and have no embedded components (condition~($\diamond$) in
\cite[3.0]{Art}). Note that it is equivalent to require these
conditions for the {\em geometric} fibres, see~\cite{EGA}~IV.4.2.7.
\end{enonce}

It is clear that for any $T\to S$, an open
subset $U$ of $X_T$ is $T$-dense if and only if for all $t$ in $T$ its
fibre $U_t$ is topologically dense in~$X_t$ (the ``no embedded
components'' condition is used here).

We can now generalise Definition~\ref{def_2.1} to the present
situation. Unlike Artin, we do not assume the immersion $W\to
X\times_S X\times_S X$ below to be of finite presentation.
\begin{defi}\label{def_strict-S-rat-grlaw}
A \emph{strict $S$-birational group law on $X/S$} is a (locally closed)
subscheme 
$W$ of $X\times_S X\times_S X$ that
satisfies the following conditions.
\begin{enumerate}
\item For every $i$ and $j$ in $\{1,2,3\}$ with $i<j$ the projection
  $\pr_{i,j}\colon W\to X\times_S X$ is an open immersion whose image,
  denoted $U_{i,j}$, is $S$-dense in $X\times_S X$ (condition~$(*)$ in
  \cite[3.0]{Art}). For each such $(i,j)$, we let $f_{i,j}\colon
  U_{i,j}\to X$ be the $S$-morphism such that $W$ is its graph. For
  every such $(i,j)$, for every $T\to S$ and for every
  $x=(x_1,x_2,x_3)$ in $X(T)^3$ the condition $x\in W(T)$ is
  equivalent to: $(x_i,x_j)\in U_{i,j}(T)$ and $x_k=f_{i,j}(x_i,x_j)$
  in $X(T)$, with $\{i,j,k\}=\{1,2,3\}$. We denote the $S$-morphism
  $f_{1,2}\colon U_{1,2}\to X$ by $(a,b)\mapsto ab$. Hence, for
  $(a,b,c)$ in $X(T)^3$ we have $(a,b,c)\in W(T)$ if and only if $(a,b)\in
  U_{1,2}(T)$ and $c=ab$ in $X(T)$.
\item For every $T\to S$, for every $a$ in $X(T)$, and for every $i$
  and $j$ in $\{1,2,3\}$ with $i<j$ the inverse images of
  $(U_{i,j})_T$ under the morphisms $(a,\id_{X_T})$ and
  $(\id_{X_T},a)\colon X_T\to (X\times_S X)_T$ are $T$-dense in~$X_T$.
\item For every $T\to S$ and all $(a,b,c)\in X(T)^3$ such that
  $(a,b)$, $(b,c)$, $(ab,c)$ and $(a,bc)$ are in $U_{1,2}(T)$, we have
  $a(bc)=(ab)c$ in~$X(T)$.
\end{enumerate}
\end{defi}

Let $W$ be a strict birational group law on~$X/S$. Then $(X,W)$ is a
\emph{group germ over $S$} as in Definition~3.1 of~\cite{Art}.

\begin{defi}
For $T\to S$, let $\calR(T)$ be the set of $T$-rational maps from
$X_T$ to itself that have a representative $(U,f)$ with $U\subset X_T$
open and $T$-dense, and $f$ an isomorphism from $U$ to an open
$T$-dense open subset $V$ of~$X_T$. As in Section~2.5 of \cite{BLR},
we call elements of $\calR(T)$ \emph{$T$-birational maps} from $X_T$
to itself.
\end{defi}

The next lemma says that every $f$ in $\calR(T)$ has a unique maximal
representative.

\begin{lemm}\label{lem_3.6}
Let $T\to S$, $U$ and $V$ be open and $T$-dense in $X_T$, and $f\colon
U\to V$ a $T$-isomorphism, and let $g\colon V\to U$ be its
inverse. Let $\Dom(f)$ and $\Dom(g)$ be their domains of
definition. Then we have $f\colon\Dom(f)\to X_T$ and
$g\colon\Dom(g)\to X_T$, and $f$ and $g$ induce inverse morphisms
between $f^{-1}\Dom(g)$ and $g^{-1}\Dom(f)$, and $U\subset
f^{-1}\Dom(g)$ and $V\subset g^{-1}\Dom(f)$.
\end{lemm}
\begin{proof}
By Proposition~\ref{prop_3.2}, $f$ and $g$ extend uniquely to
$\Dom(f)$ and $\Dom(g)$, respectively. Note that $U\subset \Dom(f)$
and $V\subset \Dom(g)$, hence $U\subset f^{-1}\Dom(g)$ and $V\subset
g^{-1}\Dom(f)$. By definition, we have $g\circ f\colon
f^{-1}\Dom(g)\to X_T$, and on $U$ this is equal to the inclusion
morphism, hence $g\circ f$ \emph{is} the inclusion morphism of
$f^{-1}\Dom(g)$ in~$X_T$. But then $g\circ f$ factors through
$\Dom(f)$, hence $f\colon f^{-1}\Dom(g)\to \Dom(g)$ factors through
$g^{-1}\Dom(f)$. So we have morphisms $f\colon f^{-1}\Dom(g)\to
g^{-1}\Dom(f)$ and $g\colon g^{-1}\Dom(f)\to f^{-1}\Dom(g)$. Then
$g\circ f$ is equal to the identity on $U$, and therefore \emph{is}
the identity, and similarly for $f\circ g$.
\end{proof}

\begin{lemm}\label{lem_G=sheaf}
For all $T\to S$, the elements of $\calR(T)$ can be composed, have a
two-sided inverse, and this makes $T\mapsto \calR(T)$ into a presheaf of
groups on~$S$. It is a sheaf for the fppf topology.
\end{lemm}
\begin{proof}
  For composition of $f$ and $g$ in $\calR(T)$ (which are equivalence
  classes), choose representatives $(U,f,U')$ and $(V,g,V')$. Then
  $U'\cap V$ is $S$-dense in $X_T$. This gives $(f^{-1}(U'\cap
  V),g\circ f,g(U'\cap V))$. Its equivalence class does not depend on
  the choices of $(U,f,U')$ and $(V,g,V')$ because of
  Proposition~\ref{prop_3.2}. For the inverse of $f$: take a
  representative $(U,f,U')$, then $f^{-1}$ is the equivalence class of
  $(U',f^{-1},U)$ (independent of choice). As composition is
  associative, and as we have the identity map, $\calR(T)$ with this
  composition is a group.

  For $f$ in $\calR(T)$ and $T_1\to T$, we get $f_1=f_{T_1}$ in
  $\calR(T_1)$ by base change as follows. Let us denote by a subscript
  $1$ the pullbacks along $T_1\to T$. Let $f$ be represented by
  $(U,f,U')$, then $U_1$ and $U'_1$ are $T_1$-dense in $X_1$ (this is
  by definition). We let $f_1$ be the equivalence class of
  $(U_1,f_1,U'_1)$; this is independent of the choice of
  $(U,f,U')$. Pullback is functorial, hence $\calR$ is a presheaf of
  groups.

  Let us prove that $\calR$ is a sheaf for the fppf topology on
  $\Sch/S$.  By Proposition~6.3.1 of~\cite{Dem1}, it suffices to
  consider a cover $q\colon S'\to S$ with $q$ faithfully flat and
  locally of finite presentation. Let $S'':=S'\times_S S'$, with the
  two projection morphisms $\pr_1$ and $\pr_2\colon S''\to S'$. Let
  $f'$ be in $\calR(S')$, such that $\pr_1^*f'=\pr_2^*f'$ in
  $\calR(S'')$. Then we have $\Dom(\pr_1^*f')=\pr_1^{-1}\Dom(f')$ by
  Proposition~\ref{prop_3.2}, and the same for~$\pr_2$. Therefore,
  $\pr_1^*\Dom(f') = \pr_2^*\Dom(f')$. Therefore, with
  $U:=q(\Dom(f'))$, which is open in $X$ because $q\colon X_{S'}\to X$
  is faithfully flat and locally of finite presentation hence open, we
  have $\Dom(f')=q^{-1}U$, and $U\subset X$ is $S$-dense because for
  every $s$ in $S$ the fibre $U_s$ is dense in~$X_s$. Hence (fully
  faithfulness of the pullback functor $q^*$ from $\Sch/S$ to the
  category of $S'$-schemes with descent datum to $S$, see
  Proposition~6.3.1(iii)+(iv) of~\cite{Dem1}) there is a unique
  $S$-morphism $f\colon U\to X$ such that $q^*f=f'\colon\Dom(f')\to
  X_{S'}$. Let $g'$ in $\calR(S')$ be the inverse of $f'$. Then we
  also have $V:=q(\Dom(g'))$ open and $S$-dense in $X$ and a unique
  $S$-morphism $g\colon V\to X$ such that $g'=q^*g$. Now the formation
  of $f^{-1}\Dom(g)$ and $g^{-1}\Dom(f)$ commutes with the base change
  $S'\to S$ by Proposition~\ref{prop_3.2}. Over $S'$ we have that $f'$
  and $g'$ are inverses on these two $S'$-dense open subsets of
  $X_{S'}$, and therefore (using again that $q^*$ is fully faithful
  and the fibre-wise criterion for $S$-denseness) $f$ and $g$ are
  inverses on $f^{-1}\Dom(g)$ and $g^{-1}\Dom(f)$ and these two opens
  are $S$-dense in~$X$. This proves that $f$ is in~$\calR(S)$.
\end{proof}

\begin{lemm}
For $T\to S$ and $a\in X(T)$, let $U_a$ be the inverse image
of $(U_{1,2})_T$ under $(a,\id_{X_T}) \colon X_T\to (X\times_S X)_T$,
and let $V_a:=(a,\id_{X_T})^{-1}(U_{1,3})_T$. Then $U_a$ and $V_a$ are
open and 
$T$-dense in $X_T$, and $(f_{1,2})_T\circ (a,\id_{X_T})\colon U_a\to
X_T$ and $(f_{1,3})_T\circ (a,\id_{X_T})\colon V_a\to X_T$ induce
inverse morphisms between $U_a$ and~$V_a$.
\end{lemm}
\begin{proof}
The $T$-density of the opens $U_a$ and $V_a$ comes from
Definition~\ref{def_strict-S-rat-grlaw}. For $T'\to T$, $b$ and $c$ in
$X(T')$, the condition $(a,b,c)\in W(T')$ is equivalent to ($(a,b)\in
U_{1,2}(T')$ and $c=f_{1,2}(a,b)$ in $X(T')$), and to ($(a,c)\in
U_{1,3}(T')$ and $b=f_{1,3}(a,c)$ in $X(T')$). But $(a,b)\in
U_{1,2}(T')$ means, by definition, that $b\in U_a(T')$. And $(a,c)\in
U_{1,3}(T')$ means that $c\in V_a(T')$.
\end{proof}

\begin{defi}\label{def_G}
For $T\to S$ and $a$ in $X(T)$, we let $\phi(a)$ denote the element of
$\calR(T)$ given by $(U_a,f_{1,2}\circ (a,\id_X),V_a)$. Hence, for
every $T'\to T$, $\phi(a)\colon U_a(T')\to V_a(T')$ is the bijection
$x\mapsto ax$. We have $\phi\colon X\to \calR$, a morphism of sheaves
on $S_{\mathrm{fppf}}$, from the sheaf of sets $X$ to the sheaf of
groups~$\calR$. We let $G$ be the subsheaf of groups of $\calR$ generated
by~$\phi(X)$.
\end{defi}

\begin{lemm}\label{lem_phi=homom}
For every $T\to S$, for all $(a,b)\in U_{1,2}(T)$ we have
$\phi(a)\circ \phi(b)=\phi(ab)$.  For $T\to S$ and $a$ in $X(T)$ we
have $\psi(a)$ in $\calR(T)$ given by $x\mapsto xa$ on some appropriate
$T$-dense open subschemes of~$X_T$. For every $T\to S$ and all $a$ and $b$ in
$X(T)$, we have $\phi(a)\circ\psi(b) = \psi(b)\circ\phi(a)$
in~$\calR(T)$. For every $T\to S$, every $g$ in $G(T)$ commutes with
every $\psi(b)$ ($b\in X(T)$).
\end{lemm}
\begin{proof}
The first statement is the associativity of the birational group
law. The statement concerning $\psi$ is proved just as for~$\phi$. The
third statement is again the associativity of the birational group
law. The last statement follows from the definition of $G$: fppf
locally on $T$, $g$ is in the subgroup of $\calR(T)$ generated by
$\{\phi(a): a\in X(T)\}$.
\end{proof}

\begin{lemm}
Let $T\to S$, $g$ be in $G(T)$ and $x\in\Dom(g)(T)$ such that $gx=x$
in $X(T)$. Then $g=\id_{X_T}$.
\end{lemm}
\begin{proof}
For every $T'\to T$ and $b\in X(T')$ such that $b\in U_x(T')$ and
$xb\in\Dom(g)(T')$, we have
\[
g(xb) = (g\circ\psi(b))x = (\psi(b)\circ g)x = xb\quad\text{in $X(T')$}.
\]
It follows that $g$ is the identity on the $T$-dense open subset of
$X_T$ that is the image of $(\id_{X_T},b)^{-1}(U_{1,2})_T$ under the
open immersion $\psi(b)\colon (\id_{X_T},b)^{-1}(U_{1,2})_T\to
X_T$. Hence $g=\id_{X_T}$.
\end{proof}

\begin{lemm}\label{lem_phi_inj}
The morphism of sheaves $\phi\colon X\to G$ is injective.
\end{lemm}
\begin{proof}
Let $T\to S$, and $a$ and $b$ be in $X(T)$, such that $\phi(a) =
\phi(b)$ in $G(T)$. Then, for all $T'\to T$ and all $x\in (U_a\cap
U_b)(T')$, we have $(a,x,ax)\in W(T')$ and $(b,x,bx)\in W(T')$. But
these two points of $W(T')$ have the same image in $(X\times_S X)(T')$
under $\pr_{2,3}$, as $ax=(\phi a)x=(\phi b)x=bx$. As $\pr_{2,3}\colon
W\to X^2$ is an open immersion, $a=b$ in $T'$. Now take $T':=U_a\cap
U_b$, which is $T$-dense in $X_T$, hence faithfully flat over $T$, and
take for $x$ the identity. Then we get $a=b$ in $X(T')$, hence $a=b$
in $X(T)$. 
\end{proof}

\begin{lemm}\label{lem_phi_open_imm}
The morphism of sheaves $\phi\colon X\to G$ is representable, and an
open immersion.
\end{lemm}
\begin{proof}
Let $T\to S$ and $g\in G(T)$. What we must show is that there is an
open subscheme $V\subset T$ and an $S$-morphism $a_V\colon V\to X$
such that $\phi(a_V)=g$ in $G(V)$ and such that for all $T'\to T$,
$a\in X(T')$ with $\phi(a)=g$ in $G(T')$, $T'\to T$ factors through
$V$ and $a=a_V$ in $X(T')$.

Let us first produce $V$ and $a_V$. We have $\Dom(g)\subset X_T$, open
and $T$-dense, and we have $(\id_{X_T},g)\colon\Dom(g)\to (X\times_S
X)_T$. This gives us the open subset $V':=(\id_{X_T},g)^{-1}(U_{2,3})_T$
of $\Dom(g)$. As $f_T\colon X_T\to T$ is open (being flat and locally
of finite presentation), we get an open subset $V$ of $T$ as $V:=f_T V'$.
On $V'$ we have
\[
a_{V'}\colon V' \stackrel{(\id_{X_T},g)}{\relbar\joinrel\relbar\joinrel\longrightarrow} (U_{2,3})_V 
\stackrel{(f_{2,3})_V}{\relbar\joinrel\relbar\joinrel\longrightarrow} X_V , \quad x\longmapsto f_{2,3}(x,gx).
\] 
Then $(a_{V'},\id_{X_T},g)$ in $X(V')^3$ is in $W(V')$, hence
$g=\phi(a_{V'})$ in $G(V')$. We claim that there is a unique $a_V$ in
$G(V)$ that is mapped to $a_{V'}$ in $G(V')$ under $f_V\colon V'\to
V$, that is, $a_{V'}=f_V^* a_V$. Now note that $f_V\colon V'\to V$ is
faithfully flat and locally of finite presentation, being the
composition of the $V$-dense open immersion of $V'$ into $X_V$, and
the morphism $X_V\to V$. Therefore (Proposition~6.3.1(iii)+(iv)
of~\cite{Dem1}) $f_V\colon V'\to V$ is a morphism of descent, that is,
the pullback functor $f_V^*$ from the category of schemes over $V$ to
the category of schemes over $V'$ with descent datum for $V'\to V$ is
fully faithful. Let $V'':=V'\times_V V'$, and let $\pr_1$ and $\pr_2$
be the two projections from $V''$ to $V'$. Then we have
$\pr_1^*a_{V'}$ and $\pr_2^*a_{V'}$ in $X(V'')$, and we have, in
$G(V'')$,
\[
\phi(\pr_1^*a_{V'}) = \pr_1^*\phi(a_{V'}) = \pr_1^*g = \pr_2^*g = 
\pr_2^*\phi(a_{V'}) = \phi(\pr_2^*a_{V'}).
\]
Lemma~\ref{lem_phi_inj} then gives $\pr_1^*a_{V'}=\pr_2^*a_{V'}$ in
$X(V'')$. But then, by the fully faithfulness of $f_V^*$, there is a
unique $a_V$ in $X(V)$ that gives $a_{V'}$ by pullback. Then we have,
in $G(V)$, $g=\phi(a_V)$, because
$f_V^*g=\phi(a_{V'})=f_V^*\phi(a_V)$.

Let us now show that $V\subset T$ and $a_V\in X_V(V)=X_V$ have the
universal property mentioned above. So let $T'\to T$, and $a\in X(T')$
such that $\phi(a)=g$ in $G(T')$. What is to be shown is that $T'\to
T$ factors through $V\subset T$ and that $a=a_V$ in $X(T')$. For this,
we may work locally on $T'$ in the fppf topology, hence we may assume
that we have $x$ in $X(T')$ such that $x\in(\Dom g)(T') $ and $x\in
(a,\id_X)^{-1}U_{1,2}$. Then we have $(a,x,gx)\in W(T')$, hence
$(x,gx)\in U_{2,3}(T')$ and $a=f_{2,3}(x,gx)$ in $X(T')$. But this
means that $T'\to T$ factors through $V'=(\id_{X_T},g)^{-1}(U_{2,3})_T$,
and hence through $V$ because $V$ is the image of $V'$ in~$T$.
\end{proof}

\begin{defi}
For $T\to S$ and $g$ in $G(T)$ let
$\phi_g:=(g{\cdot})\circ\phi_T\colon X_T\to G_T$, that is, $\phi_g$ is
$\phi_T\colon X_T\to G_T$, followed by left multiplication by $g$
on~$G_T$. For $T'\to T$ and $x$ in $X(T')$, it sends $x$ to
$g\circ(\phi x)$ in~$G(T')$.
\end{defi}

By Lemma~\ref{lem_phi_open_imm}, the $\phi_g$ are open immersions. For
$T\to S$ and $g\in G(T)$ the fibred product of $\phi_T\colon X_T\to
G_T$ and $\phi_{g^{-1}}\colon X_T\to G_T$ is their intersection
$X_T\cap (g^{-1}{\cdot})X_T$. Composing each of them with
$(g{\cdot})\colon G_T\to G_T$ gives an isomorphism $(g{\cdot})$ from
$X_T\cap (g^{-1}{\cdot})X_T$ to $(g{\cdot})X_T\cap X_T$.

\begin{lemm}\label{lem_dom_g}
The formation of $(g{\cdot})X_T\cap X_T$ from $g$ commutes with every
base change $T'\to T$. For $T\to S$, and $g\in G(T)$, we have $X_T\cap
(g^{-1}{\cdot})X_T=\Dom(g)$, $(g{\cdot})X_T\cap X_T=\Dom(g^{-1})$, and
$g$ and $g^{-1}$ are inverse morphisms between them.
\end{lemm}

\begin{proof}
The first statement follows directly from the definitions. Let us prove
the second statement. Let $t\to T$ be a geometric point. Then
$X_{t,\red}$ is a dense open subvariety of $G_{t,\red}$
(Theorem~\ref{thm_Weil}) $X_{t,\red}\cap (g^{-1}{\cdot})X_{t,\red}$ is
precisely the open subset of $X_{t,\red}$ that is mapped into
$X_{t,\red}$ under left-multiplication by~$g_t$. Therefore,
$X_{t,\red}\cap (g^{-1}{\cdot})X_{t,\red}$ equals $\Dom(g_t)$: it is
contained in $\Dom(g_t)$ because $g_t$ is a morphism on it, and the
points in its complement are mapped, by $g_t$,
outside~$X_{t,\red}$. We conclude that the open subschemes $X_T\cap
(g^{-1}{\cdot})X_T$ and $\Dom(g)$ have the same geometric fibres over
$T$, and hence are equal.
\end{proof}

\begin{lemm}\label{lem_delta_surj}
Let $\delta\colon X\times_S X\to G$ be the morphism of sheaves given by
$(a,b)\mapsto \phi(a)\circ\phi(b)^{-1}$. For $T\to S$ and $g$ in
$G(T)$, the fibre of $\delta$ over $g$ is the transpose of the graph
of $(g{\cdot})\colon X_T\cap(g^{-1}{\cdot})X_T\to (g{\cdot})X_T\cap
X_T$. The morphism $\delta$ is representable, faithfully flat and
locally of finite presentation.
\end{lemm}

\begin{proof}
The first statement is true by definition on the sets of $T'$-points for
all $T'\to T$, and therefore it is true. The second statement follows
from the first, plus the facts that
$X_T\cap(g^{-1}{\cdot})X_T=\Dom(g_T)$ is open and $T$-dense in $X_T$
and that $X\to S$ is faithfully flat and locally of finite
presentation.
\end{proof}

\begin{enonce}{Convention}\label{conventions_alg_spaces}
By an {\em algebraic space} we mean the quotient of an \'etale equivalence
relation of schemes, as in~\cite{RG}~I.5.7 or in the Stacks Project~\cite{SP}.
Thus in contrast with Knutson~\cite{Kn}, our algebraic spaces are not
necessarily quasi-separated; see also
Appendix~A of~\cite{CLO} for further comments on this point.
\end{enonce}

We can now state a variant of Theorem~3.7 of~\cite{Art}.

\begin{theo}\label{thm-alg-space}
Let $S$ be a scheme and $(X,W)$ a strict $S$-birational group law, with
$X\to S$ faithfully flat locally of finite presentation and whose
geometric fibres are separated and have no embedded components. With
the notation as above, we have the following. 
\begin{enumerate}
\item Existence:
{\rm (i)} $G\to S$ is a group algebraic space and its formation from
$(X,W)/S$ commutes with base change on~$S$.

{\rm (ii)} $\phi\colon X\to G$ is representable by $S$-dense open
  immersions, and is compatible with (rational) group laws on $X$
  and~$G$.

{\rm (iii)} The morphism $\delta\colon X\times_S X\to G$, $(a,b)\mapsto
  \phi(a)\phi(b)^{-1}$ is representable, faithfully flat and locally
  of finite presentation.
\item Uniqueness: the properties in {\rm 1.} determine $G$ and
  $\phi\colon X\to G$ up to unique isomorphism.
\item Properties: $G\to S$ is faithfully flat and locally of finite
  presentation.  If $X\to S$ is smooth (resp. quasi-compact,
  resp. with geometrically irreducible fibres), then $G\to S$ also.
\item Separation: $G\to S$ is locally separated, that is, its diagonal
is an immersion. If $X\to S$ is quasi-separated, or if $S$ is locally
noetherian, then $G\to S$ is quasi-separated. If $X\to S$ is separated,
then $G\to S$ also.
\item Representability: if $X\to S$ is of finite presentation, then
locally on~$S$ for the fppf topology, the algebraic space $G\to S$
is a scheme, faithfully flat and of finite presentation.
\end{enumerate}
\end{theo}
\begin{proof}
Let us write $X^2$ for $X\times_S X$ and $G^2$ for $G\times_S G$.

1. In Definition~\ref{def_G} we defined $G$ to be a sheaf.
Lemma~\ref{lem_phi_open_imm} says that $\phi$ is an open immersion.
By Lemma~\ref{lem_dom_g} it is $S$-dense (in the obvious sense).
Lemma~\ref{lem_phi=homom} gives the compatibility with group laws.
This proves (ii). Assertion (iii) is part of Lemma~\ref{lem_delta_surj}.
Then it follows from Artin's theorem on flat equivalence relations,
proved without quasi-separatedness condition in~\cite{SP},
Theorem~\href{http://math.columbia.edu/algebraic_geometry/stacks-git/locate.php?tag=04S6}{04S6},
that $G$ is an algebraic space. The fact that its formation commutes
with base change is true by definition.

2. Suppose that $G'$ and $\phi'\colon X\to G'$ satisfy the properties
in~1. Properties (i) and (ii) give an action by $G'$ on $X$ by
$S$-birational maps. This action is faithful, hence gives an injective
morphism of sheaves from $G'$ to~$\calR$. As morphisms to $\calR$,
$\phi$ and $\phi'$ are equal. Therefore $G$ and $G'$ are equal, being
generated by the images of $\phi$ and~$\phi'$ by property~(iii). Hence
we have an isomorphism between $G$ and $G'$ compatible with $\phi$
and~$\phi'$. There is at most one such an isomorphism because of
property~(iii).

3. Since $X\to S$ is faithfully flat locally of finite presentation,
then $X^2\to S$ also is. If moreover $X\to S$ is smooth or quasi-compact
or with geometrically irreducible fibres, then $X^2$ also. In any case,
using the presentation $\delta\colon X^2\to G$, we see that $G$ inherits
the properties.

4. Let $\delta^*\Delta_G\colon X^2\times_G X^2\to X^2\times_S X^2$
be the monomorphism of schemes obtained by pulling back the
diagonal $\Delta_G:G\to G^2$ along $\delta\times\delta$, and
$\pi=\pr_{234}\circ\delta^*\Delta_G\colon X^2\times_G X^2\lto X\times_S X^2$.
By Lemma~\ref{lem_dom_g} and Lemma~\ref{lem_delta_surj} applied with $T=X^2$
and $g=\delta$, the map $\pi$
is an open immersion. Let $\mathcal{U}\subset X\times_S X^2$ be the image
of $\pi$ and $\mathcal{V}=\pr_{234}^{-1}(\mathcal{U})\subset X^2\times_S X^2$.
What we just said means that
$X^2\times_G X^2$ is a section of $\mathcal{V}\to\mathcal{U}$. Since
sections of morphisms of schemes are immersions, this proves that
$\delta^*\Delta_G$ is an immersion, hence $\Delta_G$ is an immersion
(note that the property of being an immersion is fppf local on the
target by~\cite{SP},
Lemma~\href{http://math.columbia.edu/algebraic_geometry/stacks-git/locate.php?tag=02YM}{02YM}).
If $X\to S$ is quasi-separated, then $X^2\to S$ also and we see
using the presentation $\delta\colon X^2\to G$ that $G\to S$ is
quasi-separated. Now let us assume $S$ locally noetherian.
Let $V$ be an open affine in $S$
and $U_1$, $U_2$, $U_3$, $U_4$ open affines in $X_V$. The inverse image of
$U_1\times_SU_2\times_SU_3\times_SU_4$ in $(X\times_S
X)\times_G(X\times_S X)$ is isomorphic to an open subset of
the affine noetherian scheme $U_2\times_VU_3\times_VU_4$, and is
therefore quasi-compact, i.e. $G\to S$ is quasi-separated.

Finally we prove that $X\to S$ separated implies $G\to S$ separated.
For this, we use some notions on algebraic spaces that either have a
treatment in the existing literature (e.g. in~\cite{Kn}, \cite{SP}) or
translate immediately from the analogous notions for schemes (e.g.,
the schematic image of a morphism can be defined by descent from an
fppf presentation of the target). In order to show that
$\Delta_G\colon G\to G^2$ is a closed immersion, we may pass to a
covering of $S$ in the fppf topology; we use $T:=X\to S$, and we
denote $g\in G_T(T)$ the tautological point given by~$\phi\colon T\to
G$. Now let $D$ be the image of $\Delta_G$, let $\bar D$ be its
schematic closure inside $G\times_S G$, and let $\partial D=\bar
D\setminus D$ be the boundary.  The formation of $\partial D$ commutes
with flat base change on $G^2$, because it is the case for images of
immersions and schematic images of arbitrary morphisms. This has two
consequences. The first is that $(\partial D)_T$ is the boundary of
the image of~$\Delta_{G_T}$. Hence $(\partial D)_T$, inside
$G_T\times_T G_T$, is invariant by the action of $g$ by simultaneous
left translation ($(x,y)\mapsto(gx,gy)$), because $g$ gives an
automorphism of $\Delta_{G_T}\colon G_T\to G_T^2$ by simultaneous
left-translation.  The second consequence is that $(\partial D)\cap
X^2=\varnothing$, because that intersection is the boundary of $D\cap
X^2$ inside $X^2$, hence is empty because $X\to S$ is separated. Now
assume $\partial D\ne\varnothing$. Then there is an algebraically
closed field $k$, an $s$ in $S(k)$ and a $(x,y)\in (\partial D)(k)$
over~$s$.  Let $T_s$ denote the fiber of $T\to S$ over~$s$, it is a
$k$-scheme, locally of finite type (recall that it is~$X_s$). Then the
sets $\{t\in T_S(k):g(t)x\in X_s(k)\}$ and $\{t\in T_S(k):g(t)y\in
X_s(k)\}$ are both open and dense in $T_s(k)$, hence their
intersection is non-empty. This contradicts that $(\partial D)_T\cap
X_T^2=\varnothing$. Thus $\partial D=\varnothing$ and we are done.

5. We refer to Theorem~3.7(iii) of~\cite{Art}.
To prove that, Artin shows that for each $s$ in $S$
there is an open neighborhood $V$ of $s$ and a $T\to V$ that is
faithfully flat and of finite presentation such that $G_T$ is covered
by the open immersions $\phi_g\colon X_T\to G_T$, where $g$ varies in
a finite subset of~$G(T)$. Then $G_T$ is a scheme, faithfully flat and
of finite presentation over~$T$. The fact that Artin assumes $W\to
X^3$ to be of finite presentation is harmless; he used this assumption
only in the proof of Proposition~3.5 of~\cite{Art}, which says that
the projections from $X^2\times_G X^2$ to $X^2$ are of finite
presentation, and we have proved that result in 3. and 4. above.
\end{proof}

As a complement, we indicate with reference to the literature the state
of the art concerning the question of representability by a scheme of
the group algebraic space $G$ of Theorem~\ref{thm-alg-space}.

\begin{theo}\label{cor-alg-space}
In Theorem \ref{thm-alg-space}, assume moreover that $S$ is locally
noetherian. Then the group algebraic space $G$ over $S$
is representable by a scheme in the following cases:
\begin{enumerate}
\item $S$ has dimension $0$, e.g. the spectrum of a field
or of an artinian ring,
\item $S$ has dimension 1 and $X\to S$ is separated,
\item $X\to S$ is smooth.
\end{enumerate}
\end{theo}
\begin{proof}
1. Assume that $S$ is the spectrum of a field. Then $G$ is quasi-separated
by~\ref{thm-alg-space}.4, and the set $U$ consisting of all points of $G$
admitting a scheme-like neighborhood is topologically dense
by~\cite{RG}~5.7.7. Using finite Galois descent, one shows that~$U$
is invariant under $G$, since any finite set of points of $U$ is contained
in an affine open subscheme of $U$. In our case, $G=GU$ so that
$U=G$ is a scheme. If $S$ is local artinian, the result follows
from the previous argument since any finite subset
of a group over a local artinian ring is contained in an affine open
(Lemma~5.6.1 in~\cite{Ber}).

2. This is a result of Anantharaman, see \cite{An}~IV.4.B.

3. As discussed in~\cite{BLR} Chapter~6,
this is an application of the theorem of the square, the
quasi-projectivity of torsors in the case of a normal base as proved
in~\cite{Ra2}, and a suitable criterion for effectivity of descent
proved in~\cite{Ra1}, Theorem~4.2.
More precisely, according to Theorem~\ref{thm-alg-space},
the diagonal of $G$ is a quasi-compact immersion hence
$G$ is a smooth algebraic space in the sense used in~\cite{BLR}.
It follows that the connected component along the unit section
$G^0\to S$ is a well-defined open subspace: indeed, for group
schemes this result is \cite{Ber}, th.~3.10, (iii)~$\To$~(iv)
and the arguments of the proof in {\em loc. cit.} work
verbatim, replacing \cite{EGA}.IV.15.6.5  by~\cite{Rom}~2.2.1
for the finitely presented $W\to T$ that appears in \cite{Ber}.
Applying part (b) of \cite{BLR}~6.6/2 with $G:=G^0$ acting by
translations on the space $X:=G$ with open $S$-dense
subspace $Y:=X$ shows that $G$ is a scheme.
\end{proof}

We end this section with some remarks and examples.

\begin{rema}
1. Let $G$ be the quotient $\mathbf{G}_{\rm m}/\mathbf{Z}$ of the
multiplicative group over a field by the subsheaf generated by a section
$x\in \mathbf{G}_{\rm m}(k)$ that is of infinite order. This is a group
algebraic space whose diagonal is not an immersion
(see~\ref{thm-alg-space}.4) and hence is not representable
by a group scheme, even after base change (see~\ref{thm-alg-space}.5).

2. By~\ref{cor-alg-space}.3, if $S$ is locally noetherian
the following conditions
on a smooth group algebraic space $G\to S$ are equivalent:
$G$ comes from a strict birational group law on a scheme $X\to S$;
$G$ is representable by a scheme;
$G$ contains an $S$-dense open subspace which is a scheme.
In~\cite{Ra2}~X.14, one finds an example showing that these
conditions are not always satisfied. In this example, the base $S$
is the affine plane over a field of characteristic~$2$ and the
group algebraic space $G$ is a quotient of the square of the
additive group $(\mathbf{G}_{{\rm a},S})^2$ by an \'etale closed
subgroup.

3. We do not know an example of a strict $S$-birational group law
$(X,W)$ as in Theorem~\ref{thm-alg-space} such that $G\to S$ is not
representable by a scheme.
\end{rema}

\section{Application to N\'eron models}

Let $S$ be a Dedekind scheme (a noetherian, integral, normal scheme
of dimension $1$) with field of rational functions $K$, and let $A_K$
be a $K$-abelian variety.

A {\em model} of $A_K$ over $S$ is a pair composed of an $S$-scheme
$A$ and a $K$-isomorphism $A\times_S \Spec(K)\simeq A_K$.
Usually, one refers to such a model by the letter $A$ alone. If $A$ is
an $S$-model of $A_K$, we often say that its generic fibre "is" $A_K$.
The nicest possible model one can have is a proper smooth $S$-model,
but unfortunately this does not exist in general. In the search for
good models for abelian varieties, N\'eron's tremendous idea is to
abandon the requirement of properness, insisting on smoothness and
existence of a group structure.  He was led to the following notion.

\begin{defi}
A {\em N\'eron model} of $A_K$ over $S$ is a smooth, separated model
of finite type~$A$ that satisfies the {\em N\'eron mapping property}:
each $K$-morphism $u_K\colon Z_K\to A_K$ from the generic fibre of a
smooth $S$-scheme $Z$ extends uniquely to an $S$-morphism $u\colon
Z\to A$.
\end{defi}

Our aim is to prove that a N\'eron model exists. Note that once
existence is established, the universal property implies that the
N\'eron model $A$ is unique up to canonical isomorphism; it implies
also that the law of multiplication extends, so that $A$ is an
$R$-group scheme. Therefore it could seem that for the construction of
the N\'eron model, we may forget the group structure and recover it as
a bonus. The truth is that things go the other way round: the N\'eron
model is constructed first and foremost as a group scheme, and then
one proves that it satisfies the N\'eron mapping property.

An important initial observation is that $A_K$ extends to an abelian
scheme over the complement in $S$ of a finite number of closed points
$s$, so one can reduce the construction of the N\'eron model in
general to the construction in the local case by glueing this abelian
scheme together with the finitely many local N\'eron models (i.e. over
the spectra of the local rings $\mathcal{O}_{S,s}$). Therefore it will
be enough for us to consider the case where $S$ is the spectrum of a
discrete valuation ring $R$ with field of fractions $K$ and residue
field $k$. We fix a separable closure $k\to k^{\rm{s}}$ and a strict
henselisation $R\to R^{\mathrm{sh}}$; we have an extension of fractions fields
$K\to K^{\mathrm{sh}}$.

If $A$ is a smooth, finite type and separated model satisfying the
extension property of the above definition only
for $Z$ \'etale, we say that it is a {\em weak N\'eron
  model}. Alternatively, it is equivalent to require that $A$
satisfies the extension property for $Z=\Spec(R^{\mathrm{sh}})$, as one can
see using the fact that $R^{\mathrm{sh}}$ is the inductive limit of ``all'' the
discrete valuation rings $R'$ that are \'etale over $R$.  In contrast
with N\'eron models, weak N\'eron models are not unique since their
special fibre contains in general plenty of extraneous components, as
we shall see.  The N\'eron model will be obtained as the rightmost
scheme in the following chain (hooked arrows denote open immersions):

\bigskip
\bigskip

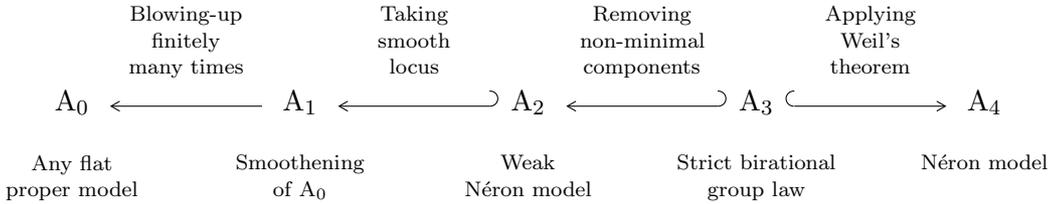
\hspace{-1cm}
\begin{tikzpicture}
\draw (0,0) node {$A_0$};
\draw (0,-1) node {{\scriptsize \begin{tabular}[t]{c} Any flat \\ proper model \end{tabular}}};
\draw [<-,>=angle 60] (0.5,-0.05) -- (2.5,-0.05);
\draw (1.5,0.8) node
{{\scriptsize \begin{tabular}{c} Blowing-up \\ finitely \\ many times \end{tabular}}};
\draw (3,0) node {$A_1$};
\draw (3,-1) node {{\scriptsize \begin{tabular}{c} Smoothening \\ of $A_0$ \end{tabular}}};
\draw [<-,>=angle 60] (3.5,-0.05) -- (5.5,-0.05);
\draw (4.5,0.8) node
{{\scriptsize \begin{tabular}{c} Taking \\ smooth \\ locus \end{tabular}}};
\draw (5.5,-0.05) arc (-90:95:0.1);
\draw (6,0) node {$A_2$};
\draw (6,-1) node {{\scriptsize \begin{tabular}{c} Weak \\ N\'eron model \end{tabular}}};
\draw [<-,>=angle 60] (6.5,-0.05) -- (8.5,-0.05);
\draw (7.5,0.8) node
{{\scriptsize \begin{tabular}{c} Removing \\ non-minimal \\ components \end{tabular}}};
\draw (8.5,-0.05) arc (-90:95:0.1);
\draw (9,0) node {$A_3$};
\draw (9,-1) node {{\scriptsize \begin{tabular}{c} Strict birational \\
group law \end{tabular}}};
\draw [->,>=angle 60] (9.5,-0.05) -- (11.5,-0.05);
\draw (10.5,0.8) node
{{\scriptsize \begin{tabular}{c} Applying \\ Weil's \\ theorem \end{tabular}}};
\draw (9.5,0.15) arc (85:270:0.1);
\draw (12,0) node {$A_4$};
\draw (12,-1) node {{\scriptsize \begin{tabular}{c} N\'eron model \\ \ \end{tabular}}};
\end{tikzpicture}

\section{N\'eron's smoothening process} \label{sc:sm_prcss}

One way to start the construction of the N\'eron model
is to choose an embedding of $A_K$ into some
projective space $\mathbf{P}^N_K$ (this is possible by a classical
consequence of the theorem of the square). Then the schematic
closure of $A_K$ inside $\mathbf{P}^N_R$ is a proper flat $R$-model $A_0$.
Then, the valuative criterion of properness implies that the
canonical map $A_0(R^{\mathrm{sh}})\to A_K(K^{\mathrm{sh}})$ is surjective. Thus if
$A_0$ happened to be smooth, it would be a weak N\'eron model of
$A_K$. It is known that the special fibre $A_0\otimes k$ is proper and
geometrically connected (see~\cite{EGA}~IV.15.5.9), unfortunately it
may be singular and even nonreduced. In order to recover smoothness at
least at integral points, in Theorem~\ref{theo_smoothening} below we
will produce a {\em smoothening} of $A_0$ as defined in the following.

\begin{defi}
Let $A$ be a flat $R$-scheme of finite type with smooth generic fibre.
A {\em smoothening} of $A$ is a proper morphism $A'\to A$ which is an
isomorphism on the generic fibres and such that the canonical map
$A'_{\mathrm{sm}}(R^{\mathrm{sh}})\to A(R^{\mathrm{sh}})$ is bijective, where $A'_{\mathrm{sm}}$ is
the smooth locus of $A'$.
\end{defi}

In order to construct a smoothening, we will repeatedly blow up $A$
along geometrically reduced closed subschemes of the special fibre
containing the specializations of the points of $A(R^{\mathrm{sh}})$ that are
"maximally singular", in a sense that we shall define soon. This leads
to consider that the natural object to start with is a pair $(A,E)$
where $E$ is a given subset of $A(R^{\mathrm{sh}})$. Note that for any proper
morphism $A'\to A$ which is an isomorphism on the generic fibres, the
set $E$ lifts uniquely to $A'(R^{\mathrm{sh}})$ and we will identify it with
its image.  The sense in which the singularity is maximal is measured
by two invariants $\delta(A,E)$ and $t(A,E)$ which we now introduce.

\begin{defi}
Let $A$ be a flat $R$-scheme of finite type with smooth generic fibre
and let $E$ be a subset of $A(R^{\mathrm{sh}})$. For each
$a\colon\Spec(R^{\mathrm{sh}})\to A$ in $E$, we set
\begin{center}
$\delta(a)=$ the length of the torsion submodule of $a^*\Omega^1_{A/R}$.
\end{center}
The integer $\delta(A,E)=\max\{\delta(a),a\in E\}\ge 0$ is called
{\em N\'eron's measure for the defect of smoothness}.
\end{defi}

It is easy to see that $\delta(a)$ remains bounded for $a\in E$, so
that $\delta(A,E)$ is finite (see \cite{BLR} 3.3/3). Moreover, this
invariant does indeed measure the failure of smoothness:

\begin{lemm}
We have $\delta(A,E)=0$ if and only if $E\subset A_{\mathrm{sm}}(R^{\mathrm{sh}})$.
\end{lemm}

\begin{proof}
Let $a\in E$ and let $d_K=\dim_{a_K}(A_K)$ and $d_k=\dim_{a_k}(A_k)$
be the local dimensions of the fibres of $A$. By the Chevalley
semi-continuity theorem, we have $d_K\le d_k$. If $\delta(a)=0$ then
$a^*\Omega^1_{A/R}$ is free generated by $d_K$ elements. Then, at the
point $a_k$, $\Omega^1_{A_k/k}$ can be generated by $d_K$ elements,
hence also by $d_k$ elements, so that $A_k$ is smooth according to
\cite{EGA}~IV$_4$.17.15.5. Being $R$-flat, the scheme $A$ is smooth at
$a_k$ and $a\in A_{\mathrm{sm}}(R^{\mathrm{sh}})$. Conversely, if $a\in
A_{\mathrm{sm}}(R^{\mathrm{sh}})$ then $\Omega^1_{A/R}$ is locally free in a
neighbourhood of $a_k$ and hence $\delta(a)=0$.
\end{proof}

Starting from a pair $(A,E)$ as above, we define geometrically reduced
$k$-subschemes $Y^1,U^1,\dots,Y^t,U^t$ of $A_k$ and the {\em canonical
  partition}
\[
E=E^1\sqcup E^2\sqcup\dots\sqcup E^t
\]
as follows:
\begin{enumerate}
\item $Y^1$ is the Zariski closure in $A_k$ of the specializations of
  the points of $E$,
\item $U^1$ is the largest $k$-smooth open subscheme of $Y^1$ where
  $\Omega^1_{A/R}|_{Y^1}$ is locally free,
\item $E^1$ is the set of points $a\in E$ whose specialization is in
  $U^1$.
\end{enumerate}
Note that $Y^1$ is geometrically reduced because it contains a
schematically dense subset of $k^{\rm{s}}$-points (see
\cite{EGA}~IV$_3$.11.10.7) and $U^1$ is dense by generic smoothness. For
$i\ge 1$, we remove $E^1\sqcup\dots\sqcup E^i$ from $E$ and we iterate
this construction.  In this way we define $Y^{i+1}$ as the Zariski
closure in $A_k$ of the specialization of the points of $E\setminus
(E^1\sqcup\dots\sqcup E^i)$, $U^{i+1}$ as the largest smooth open
subscheme of $Y^{i+1}$ where $\Omega^1_{A/R}$ is locally free, and
$E^{i+1}$ as the set of points $a\in E$ with specialization in
$U^{i+1}$.  Since $A_k$ is noetherian, there is an integer $t\ge 0$
such that $Y^{t+1}=U^{t+1}=\varnothing$ and we end up with the canonical
partition $E=E^1\sqcup E^2\sqcup\dots\sqcup E^t$.

\begin{defi}
We write $t=t(A,E)\ge 1$ for the {\em length of the canonical partition}.
\end{defi}

The crucial ingredient of the smoothening process is given by the following
lemma, due to N\'eron and Raynaud.

\begin{lemm} \label{delta_decreases}
Let $a\in E$ be such that $a_k$ is a singular point of $A_k$. Assume that
$a\in E_i$, let $A'\to A$ be the blow-up of $Y_i$, and let $a'$ be the unique
lifting of $a$ to $A'$. Then $\delta(a')<\delta(a)$.
\end{lemm}

\begin{proof}
This is an ingenious computation of commutative algebra, which we omit.
We refer to \cite{BLR} 3.3/5.
\end{proof}

For $E\subset A(R^{\mathrm{sh}})$, we denote by $E_k$ the set of
specializations of the points of $E$ in the underlying topological
space of $A_k$. We now make a definition that is tailor-made for an
inductive proof of the theorem below.

\begin{defi}
Let $A$ be a flat $R$-scheme of finite type with smooth generic fibre
and let $E$ be a subset of $A(R^{\mathrm{sh}})$. We say that a closed
subscheme $Y\subset A_k$ is {\em $E$-permissible} if it is
geometrically reduced and the set $F=Y\cap E_k$ satisfies:
\begin{enumerate}
\item $F$ lies in the smooth locus of $Y$,
\item $F$ lies in the largest open subscheme of $Y$ where $\Omega^1_{A/R}|_{Y}$
is locally free,
\item $F$ is dense in $Y$.
\end{enumerate}
We say that the blow-up $A'\to A$ with center $Y$ is {\em $E$-permissible}
if $Y$ is $E$-permissible.
\end{defi}

Recall that for any proper morphism $A'\to A$ which is an isomorphism
on the generic fibres, the set $E$ lifts uniquely to $A'(R^{\mathrm{sh}})$ and
we identify it with its image.

\begin{theo} \label{theo_smoothening}
Let $A$ be a flat $R$-scheme of finite type with smooth generic fibre
and let $E$ be a subset of $A(R^{\mathrm{sh}})$. Then there exists a morphism
$A'\to A$, a finite sequence of $E$-permissible blow-ups, such that
each point $a\in E$ lifts uniquely to a smooth point of $A'$.
\end{theo}

\begin{proof}
We proceed by induction on the integer $\delta(A,E)+t(A,E)\ge 1$. If
$\delta(A,E)=0$, then $E$ lies in the smooth locus of $A$ and no
blow-up is needed at all; this covers the initial case of the
induction. If $\delta(A,E)\ge 1$, we consider the canonical partition
$E=E^1\sqcup \dots\sqcup E^t$. The closed subscheme $Y^t\subset A_k$
is $E^t$-permissible, since $E^{t+1}=\varnothing$ means exactly that
the specializations of points of $E^t$ lie in the open subset of the
smooth locus of $Y^t$ where $\Omega^1_{A/R}|_{Y^t}$ is locally free,
and are dense in $Y^t$. Let $A'\to A$ be the blow-up of $Y^t$.
By Lemma~\ref{delta_decreases}, we have
$\delta(A',E^t)<\delta(A,E^t)$. By the inductive assumption, there
exists a morphism $A''\to A'$ which is a finite sequence of
$E^t$-permissible blow-ups such that each point of $E^t$ lifts
uniquely to a point in the smooth locus of $A''$.  If $t=1$, we are
done. Otherwise let $E''\subset A''(R^{\mathrm{sh}})$ be obtained by looking
at $E$ as a subset of $A''(R^{\mathrm{sh}})$ and removing $E^t$, and for $1\le
i\le t-1$ let $(E'')^i$ be the set $E^i$ viewed in $E''$. Since
$A''\to A$ is a sequence of $E^t$-permissible blow-ups, it does not
affect $E^1\sqcup \dots\sqcup E^{t-1}$.  In this way one sees that
$E''=(E'')^1\sqcup\dots\sqcup (E'')^{t-1}$ is the canonical partition
of $E''$, therefore $t(A'',E'')<t(A,E)$. Applying the inductive
assumption once again, we obtain a morphism $A'''\to A''$ which is a
finite sequence of $E''$-permissible blow-ups such that points of
$E''$ lift to smooth points of $A'''$. Then $A'''\to A$ is the
morphism we are looking for.
\end{proof}

\section{From weak N\'eron models to N\'eron models}

Now let us continue the construction of N\'eron models.  We started
Section~\ref{sc:sm_prcss} with the schematic closure $A_0$ of our
abelian variety $A_K$ inside some proper $R$-scheme $B$.  According to
Theorem~\ref{theo_smoothening} applied with $E=A_0(R^{\mathrm{sh}})$, there
exists a proper morphism $A_1\to A_0$ which is an isomorphism on the
generic fibre, such that the smooth locus $A_2=(A_1)_{\mathrm{sm}}$ is a weak
N\'eron model.  We now prove that weak N\'eron models satisfy a
significant positive-dimensional reinforcement of their defining
property.

\begin{prop} \label{prop_WNMP}
Let $A$ be a weak N\'eron model of $A_K$. Then $A$ satisfies the {\em
  weak N\'eron mapping property}: each $K$-rational map
$u_K\colon Z_K\dashrightarrow A_K$ from the generic fibre of a smooth
$R$-scheme $Z$ extends uniquely to an $R$-rational map
$u\colon Z\dashrightarrow A$.
\end{prop}

Note that conversely, if the extension property of the proposition is
satisfied for a smooth and separated model $A$ of finite type, then
one sees that $A$ is a weak N\'eron model by taking $Z=\Spec(R')$ for
varying \'etale extensions $R'/R$.

\begin{proof}
Since $A$ is separated, we can first work on open subschemes of $Z$
with irreducible special fibre and then glue. In this way, we reduce
to the case where $Z$ has irreducible special fibre. Then removing
from $Z$ the scheme-theoretic closure of the exceptional locus of
$u_K$, we may assume that $u_K$ is defined everywhere. Let
$\Gamma_K\subset Z_K\times A_K$ be the graph of $u_K$, let
$\Gamma\subset Z\times A$ be its scheme-theoretic closure, and let
$p\colon\Gamma\to Z$ be the first projection. On the special fibre, the
image of $p_k$ contains all $k^{\rm{s}}$-points $z_k\in Z_k$: indeed,
since $Z$ is smooth each such point lifts to an $R^{\mathrm{sh}}$-point $z\in
Z(R^{\mathrm{sh}})$ with generic fibre $z_K$, and since $A$ is a weak N\'eron
model the image $x_K=u_K(z_K)$ extends to a point $x\in A(R^{\mathrm{sh}})$,
giving rise to a point $\gamma=(z,x)\in\Gamma$ such that
$z_k=p_k(\gamma_k)$. Since the image of $p_k$ is constructible,
containing the dense set $Z_k(k^{\rm{s}})$, it contains an open set of
$Z_k$.

In particular, the generic point $\eta$ of $Z_k$ is the image of a
point $\xi\in\Gamma_k$. Since the local rings $\mathcal{O}_{Z,\eta}$
(a discrete valuation ring with the same uniformizer as $R$) and
$\mathcal{O}_{\Gamma,\xi}$ are $R$-flat and $\mathcal{O}_{Z,\eta}\to
\mathcal{O}_{\Gamma,\xi}$ is an isomorphism on the generic fibre, one
sees that $\mathcal{O}_{\Gamma,\xi}$ is included in the fraction field
of $\mathcal{O}_{Z,\eta}$.  Given that $\mathcal{O}_{\Gamma,\xi}$
dominates $\mathcal{O}_{Z,\eta}$, it follows that
$\mathcal{O}_{Z,\eta}\to \mathcal{O}_{\Gamma,\xi}$ is an
isomorphism. The schemes $Z$ and $\Gamma$ being of finite presentation
over~$R$, the local isomorphism around $\xi$ and $\eta$ extends to an
isomorphism $U\to V$ between open neighbourhoods $U\subset\Gamma$ and
$V\subset Z$. By inverting this isomorphism and composing with the
projection $\Gamma\to A$, one obtains an extension of $u_K$ to $V$.
\end{proof}

In the final step of the construction of the N\'eron model,
we make crucial use of the
group structure of $A_K$ and in particular of the existence of
invariant volume forms.

Quite generally, if $S$ is a scheme and $G$ is a smooth $S$-group
scheme of relative dimension~$d$, it is known that the sheaf of
differential forms of maximal degree
$\Omega^d_{G/S}=\wedge^d\Omega^1_{G/S}$ is an invertible sheaf that
may be generated locally by a left-invariant differential form (see
\cite{BLR} 4.2). If $G$ is commutative, left-invariant differential
forms are also right-invariant and we call them simply {\em invariant
  forms}. Thus on the N\'eron model of $A_K$, provided it exists,
there should be an invariant global non-vanishing $d$-form, also
called an {\em invariant volume form}, with $d=\dim(A_K)$.  It is the
search for such a form that motivates the following constructions.

We start by choosing an invariant volume form $\omega$ for $A_K$,
uniquely determined up to a constant in~$K^*$.  If $A$ is a model of
$A_K$ which is smooth, separated and of finite type, then all its
fibres have pure dimension $d$ and the sheaf of differential $d$-forms
$\Omega^d_{A/R}=\wedge^d\Omega^1_{A/R}$ is invertible. Moreover, if
$\eta$ is a generic point of the special fibre $A_k$, its local ring
$\mathcal{O}_{A,\eta}$ is a discrete valuation ring with maximal ideal
generated by a uniformizer $\pi$ for~$R$. Then the stalk of
$\Omega^d_{A/R}$ at $\eta$ is a free $\mathcal{O}_{A,\eta}$-module of
rank one which may be generated by $\pi^{-r}\omega$ for a unique
integer $r\in \mathbb{Z}$ called the {\em order of $\omega$ at $\eta$} and
denoted $\ord_\eta(\omega)$. If $W$ is the irreducible component with
generic point~$\eta$, this is also called the {\em order of $\omega$
  along~$W$} and denoted $\ord_W(\omega)$. Moreover, if $\rho$ denotes
the minimum of the orders $\ord_W(\omega)$ along the various
components of $A_k$, then by changing $\omega$ into
$\pi^{-\rho}\omega$ we may and will assume that $\rho=0$. A
component~$W$ with $\ord_W(\omega)=0$ will be called {\em minimal}.

In the previous sections, we saw that blowing up in a clever way
finitely many times in the special fibre of a model of $A_K$, and
removing the non-smooth locus, we obtained a weak N\'eron model
$A_2$. Now, we consider the open subscheme $A_3\subset A_2$ obtained
by removing all the non-minimal irreducible components of the special
fibre.

\begin{lemm}
The section $\omega$ extends to a global section of $\Omega^d_{A_2/R}$
and its restriction to $A_3$ is a global generator of
$\Omega^d_{A_3/R}$.
\end{lemm}

\begin{proof}
Since $A_2$ is normal and $\omega$ is defined in codimension $\le 1$,
it extends to a global section of $\Omega^d_{A_2/R}$. Now, recall that
the zero locus of a nonzero section of a line bundle on an integral
scheme has pure codimension $1$.  Thus since the restriction of
$\omega$ to $A_3$ does not vanish in codimension $\le 1$, it does not
vanish at all and hence extends to a global generator of
$\Omega^d_{A_3/R}$.
\end{proof}

Now we denote by $m_K\colon A_K\times A_K\to A_K$ the multiplication
of the abelian variety~$A_K$.

\begin{theo} \label{th_strict_group_law}
The morphism $m_K\colon A_K\times A_K\to A_K$ extends to an
$R$-rational map $m\colon A_3\times A_3\dashrightarrow A_3$ and
the $R$-rational maps $\Phi,\Psi\colon A_3\times A_3\dashrightarrow
A_3\times A_3$ defined~by
\[
\begin{array}{c}
\Phi(x,y)=(x,xy) \medskip \\
\Psi(x,y)=(xy,y)
\end{array}
\]
are $R$-birational. In other words, $m$ is an $R$-birational group law
on $A_3$.
\end{theo}

\begin{proof}
Applying the weak N\'eron mapping property
(Proposition~\ref{prop_WNMP}), we can extend $m_K$ to an $R$-rational
map $m\colon A_3\times A_2\dashrightarrow A_2$.  We wish to prove that
$m$ induces an $R$-rational map $A_3\times A_3\dashrightarrow A_3$.  Let
$D\subset A_3\times A_2$ be the domain of definition of $m$. We define
a morphism $\varphi\colon D\to A_3\times A_2$ by the formula
$\varphi(x,y)=(x,xy)$ and we view it
as a morphism of $A_3$-schemes in the obvious way. Denote by the same
symbol $\omega'$ the pullback of $\omega$ via the projection
$\pr_2\colon A_3\times A_2\to A_2$ and its restriction to $D$. We
claim that $\varphi^*\omega'=\omega'$: indeed, this holds on the
generic fibre because~$\varphi$ is an $A_3$-morphism of left
translation, so this holds everywhere by density. Now let
$\xi=(\alpha,\beta)$ be a generic point of the special fibre of
$A_3\times A_3$ and $\eta=(\alpha,\gamma)$ its image
under~$\varphi$. Let $r=\ord_\gamma(\omega)=\ord_\eta(\omega')\ge 0$.
Then $\omega'$ is a generator of $\Omega^d_{D/A_3}$ at $\xi$ and
$\pi^{-r}\omega'$ is a generator of $\Omega^d_{A_3\times A_2/A_3}$ at
$\eta$.  It follows that $\varphi^*(\pi^{-r}\omega')=b\omega'$ for
some germ of function $b$ around $\xi$. Since
$\varphi^*(\pi^{-r}\omega')=\pi^{-r}\omega'$, this implies that $r=0$
hence $\eta\in A_3\times A_3$. This shows that the set of irreducible
components of the special fibre of $A_3\times A_3$ is mapped into
itself by $\varphi$.  Setting $U=D\cap (A_3\times A_3)$ we obtain
morphisms $\varphi\colon U\to A_3\times A_3$ and
$m=\pr_2\circ\varphi\colon U\to A_3$ that define the sought-for rational
maps.  Proceeding in the same way with the morphism $\psi\colon D\to
A_3\times A_2$ defined by $\psi(x,y)=(xy,y)$, we see that it also
induces a morphism $\psi\colon U\to A_3\times A_3$.  In this way we obtain
the $R$-rational maps $m,\Phi,\Psi$ of the theorem.

In order to prove that $\Phi$ induces an isomorphism of $U$ onto an
$R$-dense open subscheme, we show that $\varphi\colon U\to A_3\times
A_3$ is an open immersion.  We saw above that the map
\[
\varphi^*\Omega^d_{A_3\times A_3/A_3}\lto \Omega^d_{U/A_3}
\]
takes the generator $\omega'$ to itself, so it is an isomorphism. This
map is nothing else than the determinant of the morphism
\[
\varphi^*\Omega^1_{A_3\times A_3/A_3}\lto \Omega^1_{U/A_3}
\] 
on the level of $1$-forms which thus is also an isomorphism. It
follows that $\varphi$ is \'etale, and in particular
quasi-finite. Since it is an isomorphism on the generic fibre, it is
an open immersion by Zariski's Main Theorem
(\cite{EGA}~IV$_3$.8.12.10). One proves the required property for
$\Psi$ in a similar way.
\end{proof}

Let $U$ be the domain of definition of the $R$-rational map $m\colon
A_3\times A_3\dashrightarrow A_3$. Then the graph~$W$ of $m|_U\colon
U\to A_3$ is a strict $R$-rational group law on $A_3/R$ in the sense
of Definition~\ref{def_strict-S-rat-grlaw}. It follows from
Theorems~\ref{thm-alg-space} and~\ref{cor-alg-space} that
there exists an open $R$-dense immersion of~$A_3$ into a smooth
separated $R$-group scheme of finite type $A_4$. The last thing we
wish to do is to check that $A_4$ is the N\'eron model of~$A_K$:

\begin{prop}\label{prop_univ_property}
The group scheme $A_4$ is the N\'eron model of $A_K$, that is, each
$K$-morphism $u_K\colon Z_K\to A_K$ from the generic fibre of a smooth
$R$-scheme $Z$ extends uniquely to an $R$-morphism $u\colon Z\to A_4$.
\end{prop}

\begin{proof}
Let us consider the $K$-morphism $\tau_K\colon Z_K\times A_K\to A_K$
defined by $\tau_K(z,x)=u_K(z)x$. Applying the weak N\'eron mapping
property, this extends to an $R$-rational map $\tau_2\colon Z\times
A_2\dashrightarrow A_2$. In a similar way as in the proof
of~\ref{th_strict_group_law}, one proves that the induced $R$-rational
map $Z\times A_2\dashrightarrow Z\times A_2$ defined by $(z,x)\mapsto
(z,\tau_2(z,x))$ restricts to an $R$-rational map $Z\times
A_3\dashrightarrow Z\times A_3$. Since $A_3$ is $R$-birational to
$A_4$, the latter may be seen as an $R$-rational map $Z\times
A_4\dashrightarrow Z\times A_4$.  Composing with the second
projection, we obtain an $R$-rational map $\tau_4\colon Z\times
A_4\dashrightarrow A_4$ extending the map $\tau_K$. By Weil's theorem
on the extension of rational maps from smooth $R$-schemes to smooth
and separated $R$-group schemes (\cite{BLR} 4.4/1), the latter is
defined everywhere and extends to a morphism. Restricting $\tau_4$ to
the product of $Z$ with the unit section of $A_4$, we obtain the
sought-for extension of $u_K$. The fact that this extension is unique
follows immediately from the separation of $A_4$.
\end{proof}

\begin{rema}
Raynaud proved that the N\'eron model $A_4$ is quasi-projective over $R$.
In fact, one knows that there exists an ample invertible sheaf
$\mathcal{L}_K$ on $A_K$. Raynaud proved that there exists an integer~$n$ such
that the sheaf $(\mathcal{L}_K)^{\otimes n}$ extends to an $R$-ample invertible
sheaf on $A_4$, see~\cite{Ra2}, theorem VIII.2.
\end{rema}

\bigskip

\noindent {\bf Acknowledgements.}
We thank Ariane M\'ezard for fruitful discussions we had with her on the
topics discussed in this article. For example, we owe to her the inclusion
of remark~\ref{rem_extensions} on extensions of algebraic groups.
We also wish to thank M. Raynaud, Q.~Liu, and
C.~P\'epin for various interesting comments on a preliminary version of this
article. We are grateful to Ph.~Gille and P.~Polo for inviting us to take
part in the reflection on SGA3.

\bigskip

\def\refname{References}

\end{document}